\documentclass[11pt,reqno]{amsart}
\usepackage{amsfonts}
\usepackage{amsmath}
\usepackage{amssymb}
\usepackage{multicol}
\usepackage{stmaryrd}
\usepackage{cite}
\usepackage{epsfig}
\usepackage{color}
\usepackage{graphics}
\usepackage{graphicx}
\usepackage{epstopdf}
\usepackage{multicol,graphics}
\newcommand\bes{\begin{eqnarray}}
\newcommand\ees{\end{eqnarray}}
\newcommand\R{\mathbb R}
\newtheorem{theorem}{Theorem}[section]
\newtheorem{lemma}[theorem]{Lemma}

\newtheorem{remark}[theorem]{Remark}

\numberwithin{equation}{section}

\begin{document}

\title[Accelerations of Propagation in a chemotaxis-growth system]{Acceleration of Propagation in a chemotaxis-growth system with slowly decaying initial data}

\author[Zhi-An Wang, Wen-Bing Xu]{Zhi-An Wang$^1$ and Wen-Bing Xu$^{2,*}$\\
}
\thanks{\hspace{-.53cm}
$^1$ Department of Applied Mathematics, The Hong Kong Polytechnic University, Hung Hom, Kowloon, Hong Kong; mawza@polyu.edu.hk\\
$^2$ School of Mathematical Sciences, Capital Normal University, Beijing 100048,  People's Republic of China; xuwb15@lzu.edu.cn
}

\begin{abstract}
In this paper, we study the spatial propagation dynamics of a parabolic-elliptic chemotaxis system with logistic source which reduces to the well-known Fisher-KPP equation without chemotaxis.  It is known that for fast decaying initial functions, this system has a finite spreading speed. For slowly decaying initial functions, we show that the accelerating propagation will occur and chemotaxis does not affect the propagation mode determined by slowly decaying initial functions if the logistic damping is strong, that is,  the system has the same upper and lower bounds of the accelerating propagation as for the classical Fisher-KPP equation. The main new idea of proving our results is the construction of auxiliary equations to overcome the lack of comparison principle due to chemotaxis.

\vspace{1em}

\end{abstract}
\maketitle
\section{Introduction}
Chemotaxis, as a strategy of migration, describes the  directional  movement of cells along a chemical concentration gradient. It was well-known that this process can promote the rapid propagation of bacterial populations into previously unoccupied territories (cf. \cite{budrene1995dynamics, Adler66, Adler69, saragosti2011directional, liu2011sequential}).  The propagation of migrating bands of bacterial chemotaxis was first observed in the experiment by Adler \cite{Adler66} and the first mathematical model was proposed by Keller and Segel \cite{KS-tws} to recover the migrating bands of bacterial chemotaxis by a pair of reaction-diffusion-convection equations (nowadays well-known as singular Keller-Segel system) as follows
\begin{equation}\label{ks}
\left\{\begin{array}{ll}
u_{t}=\Delta u-\chi \nabla \cdot(\frac{u}{v} \nabla v), \\
v_{t}=\varepsilon \Delta v-u v^{m},
\end{array}\right.
\end{equation}
 where $u(x,t)$ denotes the bacterial density and $v(x,t)$ the chemical (oxygen) concentration at position $x$ and at time $t>0$, respectively. $\chi>0$ is the chemotactic coefficient and $\varepsilon\geq 0$ denotes the chemical diffusivity. The Keller-Segel System \eqref{ks} has attracted extensive studies generating  a large number of beautiful mathematical results on the existence and stability of traveling wave solutions (cf. \cite{Odell75, LuiWang, NagaiIkeda, Chae, Wang-TWS-DCDSBrev, LW09, LLW, Davis}) as well as global solvability (cf. \cite{hou2019convergence, HLWW, LPZ, MWZ}). In order to generate traveling bands, the Keller-Segel system \eqref{ks} requires a singular chemotactic sensitivity for sufficiently small chemical concentration, which is, however, unrealistic since cells cannot perform chemotaxis when concentrations fall below detectable values (cf. \cite{narla2021traveling}). Furthermore, this model neglected cell growth, a substantial factor in the expansion process (cf. \cite{narla2021traveling}). Subsequently various models including cell growth but without singular sensitivity have been proposed and tested against numerical simulations (cf. \cite{landman2005diffusive, lauffenburger1984traveling, lapidus1978model, tindall2008overview}). In this paper we are concerned with the spatial propagation dynamics of the following classical chemotaxis system with logistic growth
\begin{equation}\label{1.0}
\left\{
\begin{aligned}
&u_t= \Delta u-\chi \nabla \cdot(u\nabla v)+ u(a-bu),&&x\in \Omega,~t>0,\\
&\tau v_t=\Delta v-\lambda v+\mu u,&&x\in  \Omega,~t>0,
\end{aligned}
\right.
\end{equation}
where $\chi, a, b, \lambda$ and $\mu$ are positive constants,  $\tau$ is a nonnegative constant, $u(x,t)$ denotes the cell density and $v(x,t)$ the chemical concentration at position $x$ and at time $t>0$. The system \eqref{1.0} models the movement of cells directed by the higher concentration of chemoattractant emitted from cells. The constant $\chi$ is called the chemotactic coefficient, $\lambda$ is the degradation rate of chemoattractant, $\mu$ is the rate by which cells produce chemoattractant, the constant $1/\tau$ in the case $\tau>0$ measures the diffusion rate of  chemoattractant. It has been numerically demonstrated in \cite{painter2011spatio} that the model \eqref{1.0} can produce various intricate patterns including traveling waves.
When $b$ is suitably large, the global existence and stabilization of solutions to \eqref{1.0} can be ensured in $\R^n (n\geq 2)$ or in a bounded domain of $\R^n (n\geq 2)$ with Neumman boundary conditions (cf. \cite{Winkler2010, OsakiTsujikawaYagiMimura2002} for  $\tau>0$ and \cite{TelloWinkler2007,SS2017,IssaShen2017} for  $\tau=0$, and references therein).

Except the global solvability and stabilization, the spatial propagation dynamics  of \eqref{1.0} is another interesting research topic and not many results are available in the literature. In this paper, we will investigate the spatial propagation dynamics of \eqref{1.0}  in $\R$ with $\tau=0$, namely
\begin{equation}\label{1.1}
\left\{
\begin{aligned}
&u_t= u_{xx}-\chi (uv_x)_x+u(a-bu),&&x\in\mathbb R,~t>0,\\
&0=v_{xx}-\lambda v+\mu u,&&x\in\mathbb R,~t>0,\\
&u(x,0)=u_0(x),&&x\in\mathbb R.
\end{aligned}
\right.
\end{equation}
Throughout the paper, we assume  that the  initial   function satisfies
\begin{equation}\label{1.2}
u_0\in C(\mathbb R),~u_0(x)>0~\text{for}~x\in\mathbb R,~\liminf\limits_{x\rightarrow-\infty}u_0(x)>0,~\text{and}~ u_0(x)\rightarrow 0~\text{as}~x\rightarrow +\infty.
\end{equation}
There are two classes of initial function $u_0$ that are commonly used  in the literature as follows:
\begin{itemize}
  \item[(a)] \textbf{fast decaying initial function}, namely, there is   $\beta>0$ and $C>0$ such that
      \begin{equation}\label{1.98}
      u_0(x)\leqslant C e^{-\beta x}~\text{for   large}~x,
      \end{equation}
  \item[(b)]  \textbf{slowly  decaying initial function}, namely,  there is a large constant $\xi_0$ such that
      \begin{equation}\label{1.10}
      u_0\in C^2([\xi_0,+\infty)),~u_0'\leqslant0~\text{in}~[\xi_0,+\infty),~\text{and}~u_0''(x)/ u_0(x)\rightarrow 0~\text{as}~x\rightarrow +\infty.
      \end{equation}
\end{itemize}
By Lemma \ref{lem9.0} (ii) below, when $u_0$ satisfies \eqref{1.2} and \eqref{1.10}, $u_0$ decays more slowly than any exponentially decaying function as $x\rightarrow+\infty$, namely,
\[
\forall~\kappa,~\exists~x_\kappa~\text{s.t.}~u_0(x)\geqslant e^{-\kappa x}~~\text{for all}~~x\in[x_\kappa,+\infty).
\]

In the case $\chi=0$,   \eqref{1.1} becomes the following well-known Fisher-KPP equation
\begin{equation}\label{1.97}
\left\{
\begin{aligned}
&u_t= u_{xx}+u(a-bu),&&x\in\mathbb R,~t>0,\\
&u(x,0)=u_0(x),&&x\in\mathbb R.
\end{aligned}
\right.
\end{equation}
The spatial propagation dynamics of \eqref{1.97} with different initial functions  has been extensively studied as one of the prevailing research topics in the past few decades. For example,
for  fast  decaying initial functions, if $u_0(x)\leqslant c_0 e^{- \sqrt{a} x}$ with $c_0>0$ for   large $x$, then \eqref{1.97} has a spreading speed $2\sqrt{a}$ (see \cite{kolmogorov1937study,AW1975,AW1978}), and if there are two constants $c_1,c_2\in(0,+\infty)$ such that  $c_1\leqslant u_0(x) e^{\kappa x}\leqslant c_2 $ with $\kappa\in(0,\sqrt{a})$ for  large $x$, then \eqref{1.97} has a spreading speed $\kappa+a/\kappa$ (see \cite{McKean1975,Sattinger1976,BootyHabermanMinzoni1993,Lau1985,HamelGadin2012}),
where a constant  $c^*>0$ is called a \emph{spreading speed}  if $u(x,t)$ satisfies that
\[
\left\{
\begin{aligned}
&\lim\limits_{t\rightarrow\infty}\sup\limits_{x\geqslant ct} |u(x,t)|=0\quad \text{for any}~c>c^*,\\
&\liminf\limits_{t\rightarrow\infty}\inf\limits_{x\leqslant ct}u(x,t)>0\quad \text{for any}~c<c^*.
\end{aligned}
\right.
\]
We refer to \cite{Weinberger1982,Lui1989,LiangZhao2007,LiangZhao2010} for  more results on the  spreading speed of discrete-time recursion equations which can include \eqref{1.97} as a  special  example.
On the other hand, for slowly decaying initial functions,  as shown in \cite{HamelRogues2010} by Hemal and Roques,  \eqref{1.97} has a new  propagation mode-\emph{acceleration propagation}, which is quite different from the propagation resulting from the fast  decaying initial functions. They provided an accurate description locating the moving level set
$$E_{\omega}(t)=\{x\in\mathbb R, u(x,t)=\omega\}$$
by showing that for any $\gamma_1,\gamma_2>0$, $\epsilon\in(0,a)$, and $\omega\in(0,a/b)$, there exists $T>0$ such that
\begin{equation}\label{1.90}
E_{\omega}(t)\subseteq u_0^{-1}\left\{\big[ \gamma_1e^{-(a+\epsilon)t}, \gamma_2e^{-(a-\epsilon)t}\big]\right\}~\text{for all}~t\geqslant T,
\end{equation}
where $u_0^{-1}\{A\}=\{x\in\mathbb R, u_0(x)\in A\}$ denotes  the inverse image of $u_0$ from the set $A$.
We refer to \cite{Alfaro2017,Henderson2016,HeWuWu2017} for the acceleration propagation results recently developed for more general reaction-diffusion equations.
We also refer to \cite{CabreRoquejoffre2013,CoulonRoquejoffre2012,FelmerYangari2013} for the acceleration propagation in fractional diffusion equations and \cite{Garnier2011,AlfaroCoville2017,FinkelshteinTkachov2019} for nonlocal dispersal equations.

Compared  to the Fisher-KPP equation,   the results on  the spatial propagation dynamics of  \eqref{1.1} with $\chi>0$ are much less. It was first shown in \cite{ryzhik2008traveling} that for any $0<\chi<1$, there is a wave speed $c\in [2, 2+\frac{\chi}{1-\chi}]$ such that \eqref{1.1} with $a=b=1$ admits traveling wave solutions connecting $(1,1)$ to $(0,0)$. Salako and Shen first studied the Cauchy problem of \eqref{1.1}  in \cite{SS2017} and identified the minimal wave speed of traveling wave solutions in \cite{Salako-DCDS-2017}. When $u_0$ is a fast decaying initial function, they further studied the existence of spreading speed of \eqref{1.1} in \cite{SSX2019}. To summarize their main results in \cite{SSX2019}, we introduce a set
\[ K=\bigg\{z\in\mathbb R~\big|~1+\frac{(z-\sqrt{\lambda})_+}{2(z+\sqrt{\lambda})}\leqslant \frac{b}{\chi\mu}\bigg\},
\]
which is equivalent to
\[
K=\left\{
\begin{aligned}
&\varnothing,&&\text{when}~~b<\chi\mu,\\
&\left(-\infty,~\frac{2b-\chi\mu}{3\chi\mu-2b}\sqrt{\lambda}\right),&&\text{when}~~\chi\mu\leqslant b<\frac{3}{2}\chi\mu,\\
&\mathbb R,&&\text{when}~~b\geqslant \frac{3}{2}\chi\mu.
\end{aligned}
\right.
\]
We denote the spreading speed of \eqref{1.1} by a positive constant $c^*$. When $b\geq \chi \mu$ and $u_0$ is a fast decaying initial function, it was shown in \cite{SSX2019} that $c^*=2\sqrt{a}$ if $\sqrt{a}\in K$ and $u_0(x)=0$ for large $x$, while $c^*=\kappa+a/\kappa$ if $u_0(x)\rightarrow e^{-\kappa x}$ as $x\rightarrow+\infty$ with   $\kappa\in(0,\sqrt{a})\cap K$. Compared to the results for the Fisher-KPP equation \eqref{1.97}, it was found that the chemotaxis neither asymptotically speeds up nor  slows down the spreading speed for fast decaying initial functions when the chemotactic sensitivity is weak (i.e. $\chi\leq \frac{b}{\mu}$). However the case of strong chemotaxis (i.e. $\chi>\frac{b}{\mu}$) was left open. Similar but weaker results for the parabolic-parabolic version of  \eqref{1.1} have been obtained in \cite{salako2018existence, salako2020traveling}.

If we solve $v$ in terms of $u$ by the constant of variations,  we can reformulate the system \eqref{1.1} into a scalar Fisher-KPP type equation with non-local advection as follows
\begin{equation}\label{1.92}
u_t+[(K \ast u)u]_x=u_{xx}+u(1-u),
\end{equation}
where we have assumed $a=b=1$ without loss of generality and
\begin{equation}\label{K}
K(x)=-\frac{1}{2}\chi\mu~\text{sign}(x) e^{-\sqrt{\lambda}|x|}, \ x\in \mathbb{R}.
\end{equation}
In  a recent work \cite{HamelHenderson2020}, the authors studied the spatial propagation dynamics of \eqref{1.92}  with more general kernel function $K$ for non-zero compactly supported initial function.
When $K$ is large negative chemotaxis, the acceleration propagation driven  by the ``heavy-tailed''   $K$ was obtained  in \cite{HamelHenderson2020} asserting that if $K\in L^p(\mathbb R)$ with $p>1$, the position of the ``front'' is of order $O(t^p)$ if $p<\infty$ and $K(x)\geqslant(1+x)^{-\alpha}$ with $\alpha\in(0,1)$,
and of order $O(e^{\lambda t})$ for some $\lambda>0$ if $p=+\infty$ and $K(+\infty)>0$. On the other hand, when $K\in L^1(\mathbb R)$ and $K=\bar K^{-1}$ for some kernel $\bar K\in W^{1,1}(\mathbb R)$, namely $K$ is ``light-tailed'', only explicit upper and lower bounds on the spreading speed were obtained. Hence in the special case that $K(x)$ of $L^1$-class given in \eqref{K}, the results of \cite{SSX2019} presented a more refined dynamics by finding the precise spreading speed.

As recalled above, we find that if $u_0$ is a slowly decaying initial function, there
is no result about the spatial propagation dynamics of  \eqref{1.1} with $\chi\ne0$, and the question whether or not the acceleration propagation occurs and how the chemotaxis affects (speeds  up  or  slows down) the spatial propagation dynamics of  \eqref{1.1} is unknown.
We shall explore these questions in this paper and our main results are stated in the following theorem.
\begin{theorem}\label{th1}
Assume $b>2\chi\mu$ and $u_0$ satisfies \eqref{1.2} and \eqref{1.10}. Then the following results hold.
\begin{itemize}
  \item[(i)] The solution of \eqref{1.1} satisfies
     \[
     \forall~ t\geqslant0,~\lim_{x\rightarrow+\infty} u(x,t)=0,~\text{and}~\liminf_{x\rightarrow-\infty}u(x,t)\rightarrow a/b~\text{as}~t\rightarrow+\infty.
     \]
  For any $\omega \in(0,a/b)$, there exists $T_\omega\geqslant0$ such that the set $E_{\omega}(t)$ with $t\geqslant T_{\omega}$ is compact and nonempty, where
     $E_{\omega}(t)=\{x\in\mathbb R, u(x,t)=\omega\}$.
  \item[(ii)] For any $\epsilon\in(0,a)$, $\gamma_1>0$, and $\gamma_2>0$, if $\zeta(t)$ and $\eta(t)$   satisfy
      \[
      u_0(\zeta(t))=\gamma_1 e^{-(a+\epsilon)t},~~
      u_0(\eta(t))=\gamma_2 e^{-(a-\epsilon)t}\quad\text{for}~t~\text{large enough},
      \]
      then for any $\omega\in(0,a/b)$, there is a constant $T\geqslant T_{\omega}$ such that
      \[
      E_\omega(t)\subseteq [\eta(t),\zeta(t)]~\text{for all}~t\geqslant T.
     \]
  \item[(iii)]  For any $\omega\in(0, a/b)$, we have that
  \[
  \lim\limits_{t\rightarrow+\infty}\frac{\inf\{E_\omega(t)\}}{t}=+\infty.
  \]
\end{itemize}
\end{theorem}

\noindent {\bf Main proof ideas}. The main difficulty of proving Theorem \ref{th1} lies in the failure of comparison principle of \eqref{1.1} due to the presence of chemotaxis which generates  a cross diffusion. To overcome this obstacle, we construct the following two auxiliary equations for which the comparison principle  holds (see Lemma \ref{le2.1})
\begin{align}
&w_t=w_{xx}+\frac{\chi\mu L}{\sqrt{\lambda}} |w_x|+ w(a-(b-\chi\mu) w), \label{1.94}\\
&w_t=w_{xx}-\frac{\chi\mu L}{\sqrt{\lambda}} |w_x|+ (a-\epsilon/2)w-Mw^{1+\delta},  \label{1.93}
\end{align}
where $L$, $M$, and $\delta$ are some appropriate positive constants. By some estimates, we find that the solution of \eqref{1.1} is  a lower solution of \eqref{1.94} but an upper solution of \eqref{1.93}.
Then by constructing  an upper solution $\bar w$ of \eqref{1.94} and a lower solution $\underline w$ of  \eqref{1.93}, we can employ  comparison principles to \eqref{1.94} and \eqref{1.93} and conclude that $\underline w\leqslant u\leqslant \bar w$. This is the primary new idea of this paper to overcome  the lack of comparison principle in \eqref{1.1} and hence to prove Theorem \ref{th1}.
\medskip

We give several remarks to further highlight some results in Theorem \ref{th1}.
\begin{remark}\em{
Theorem \ref{th1} (ii) provides a lower bound $\eta(t)$ and an upper bound $\zeta(t)$ of   $E_{\omega}(t)$   when $t$ is large enough.
Moreover, for any  $\epsilon\in(0,a)$, $\gamma_1,\gamma_2>0$, and $\omega \in(0,a/b)$, it follows from Theorem \ref{th1} (ii) that
\begin{equation}\label{1.91}
E_\omega(t)\subseteq [\eta(t),\zeta(t)]\subseteq u_0^{-1}\left\{\big[\gamma_1 e^{-(a+\epsilon)t},\gamma_2 e^{-(a-\epsilon)t}\big]\right\}~\text{for}~t~\text{large enough}.
\end{equation}
Theorem \ref{th1} (iii) shows that the average moving speed of the level set $E_{\omega}(t)$ tends to infinity as $t\rightarrow+\infty$, namely, the acceleration propagation occurs. Therefore, compared to \cite{HamelRogues2010}, our results show that if chemotaxis is not strong (i.e. $\chi<\frac{ b}{2\mu}$), then it does not affect the spatial propagation dynamics of \eqref{1.1} with slowly decaying initial functions, in the sense that the estimates of upper and lower bounds of the acceleration propagation are the same as those for the Fisher-KPP equation \eqref{1.97} (see \eqref{1.90}). However, if chemotaxis strength is strong (i.e. $\chi\geq\frac{ b}{2\mu}$), it is unknown whether it has any effect on the spatial propagation dynamics of  \eqref{1.1}.
}
\end{remark}

\begin{remark}\em{
As stated in   \cite{HamelRogues2010}, \eqref{1.90}  provides
explicit upper and lower bounds locating the moving level set $E_\omega(t)$ of \eqref{1.97} for different  examples  of $u_0$. The results in Theorem \ref{th1} (ii) and \eqref{1.91} show  that these examples also work for \eqref{1.1}, which are summarized as follows  for the purpose of lucid presentations of  different accelerating degrees for different $u_0$.
\begin{itemize}
  \item[(i)] If $u_0(x)=C e^{-p x/\ln x}$ with  $p,C>0$ for large $x$, then
      \[
       \min E_\omega(t)\sim \max E_\omega(t)\sim  a p^{-1} t\ln t~\text{as}~t\rightarrow+\infty.
      \]
  \item[(ii)] If $u_0(x)=C e^{-px^{q}}$ with   $q\in(0,1)$ and $p,C>0$ for large $x$, then
      \[
      \min E_\omega(t)\sim \max E_\omega(t)\sim \left(a/p\right)^{1/q}t^{1/q}~\text{as}~t\rightarrow+\infty.
      \]
  \item[(iii)] If $u_0(x)=C x^{-p}$ with  $p,C>0$ for large $x$, then
      \[
      \ln(\min E_\omega(t))\sim \ln(\max E_\omega(t))\sim ap^{-1}t~\text{as}~t\rightarrow+\infty.
      \]
  \item[(iv)] If $u_0(x)=C(\ln x)^{-p}$ with  $p,C>0$ for large $x$, then
      \[
      \ln\left(\ln(\min E_\omega(t))\right)\sim \ln\left(\ln(\max E_\omega(t))\right)\sim ap^{-1}t~\text{as}~t\rightarrow+\infty.
      \]
\end{itemize}
}
\end{remark}

The rest of this paper is organized as follows. In Section 2, we present some preliminary results including the global existence  of the classical solution of \eqref{1.1}, some properties of initial function, and the  comparison principle of the auxiliary equations \eqref{1.94} and \eqref{1.93} .
In Section 3, we prove Theorem \ref{th1} by estimating the upper bound and the lower bound of the moving level set.

\section{preparatory results}

In this section, we present some preliminary results.
We first state the results about  the global existence and asymptotic behavior of classical solutions of \eqref{1.1}.
Denote
\[
C_\text{unif}^b(\mathbb R)=\{u\in C(\mathbb R)~|~u~\text{is uniformly continuous in $x\in\mathbb R$ and $\sup_{x\in\mathbb R}|u(x)|<\infty$}\},
\]
which is equipped with the norm
\[
\|u\|_{\infty}=\sup_{x\in\mathbb R}|u(x)|.
\]
For  $0<\nu<1$, denote
\[
C_\text{unif}^{\nu}(\mathbb R)=\bigg\{u\in C_\text{unif}^b(\mathbb R)| \sup\limits_{x,y\in\mathbb R,x\neq y}\frac{|u(x)-u(y)|}{|x-y|^{\nu}}<\infty\bigg\}
\]
equipped with the   norm
\[
\|u\|_{\infty,\nu}=\sup_{x\in\mathbb R}|u(x)|+\sup\limits_{x,y\in\mathbb R,x\neq y}\frac{|u(x)-u(y)|}{|x-y|^{\nu}}.
\]
For   $0<\theta<1$, denote
\[
\begin{aligned}
&C^\theta((t_1,t_2),C_\text{unif}^{\nu}(\mathbb R))\\
&=\{u\in C((t_1,t_2),C_\text{unif}^{\nu}(\mathbb R))|~u(t)~\text{is locally H\"{o}lder continuous with exponent $\theta$}\}.
\end{aligned}
\]
\begin{lemma}[Salako and Shen \cite{SS2017}]\label{lem9.2}
For any nonnegative initial function $u_0\in C_{\rm unif}^b(\mathbb R)$,  there exists  $T_{\rm max}\in (0,+\infty]$ such that  \eqref{1.1} has a unique nonnegative classical solution $(u,v)$ satisfying
\begin{equation}\label{1.11}
u\in C([0,T_{\rm\max}), C_{\rm unif}^b(\mathbb R))\cap C^1((0,T_{\rm\max}), C_{\rm unif}^b(\mathbb R)),
\end{equation}
such that $u(\cdot,t) \rightarrow u_0$  in $C_{\rm unif}^b(\mathbb R)$ as $t\rightarrow 0^+$ and
\[
u,~u_x,~u_{xx},~u_t\in C^\theta((0,T_{\rm\max}),C_{\rm unif}^{\nu}(\mathbb R))
\]
for $0<\theta\ll 1$ and $0<\nu\ll 1$.
If $b>\chi\mu$, then the classical solution is global {\rm(}namely, $T_\text{max}=+\infty${\rm)} and uniformly bounded in time. Moreover, if $b>2\chi\mu$ and $\inf_{x\in\mathbb R}u_0(x)>0$, then
\begin{equation}\label{1.12}
\left\|u(\cdot,t)-\frac{a}{b}\right\|_{\infty}+\left\|v(\cdot,t)-\frac{a\mu}{b\lambda}\right\|_{\infty}\rightarrow 0~\text{as}~t\rightarrow +\infty.
\end{equation}
\end{lemma}



Now we show that when $b>2\chi\mu$, \eqref{1.1} has  no other equilibrium
$(\phi(x), \psi(x))$ satisfying  $\inf_{x\in\mathbb R}{\phi(x)}>0$ except $(\phi,\psi)\equiv (\frac{a}{b},\frac{a\mu}{b\lambda})$.
\begin{lemma} \label{lem9.1} Consider the following  system
\begin{equation}\label{2.11}
\left\{
\begin{aligned}
&\phi''-\chi (\phi \psi')'+\phi(a-b\phi)=0,&&x\in\mathbb R,\\
&\psi''-\lambda \psi+\mu \phi=0,&&x\in\mathbb R.
\end{aligned}
\right.
\end{equation}
If $b>2\chi\mu$,  $(\phi,\psi)\equiv (\frac{a}{b},\frac{a\mu}{b\lambda})$ is the unique solution of  \eqref{2.11} in $C_{\rm unif}^b(\mathbb R)$ such that $\inf_{x\in\mathbb R}{\phi(x)}>0$.
\end{lemma}
\begin{proof}
We suppose that  $(\phi,\psi)$ is a  solution of the system  \eqref{2.11} in $C_{\rm unif}^b(\mathbb R)$ with $\inf_{x\in\mathbb R}{\phi(x)}>0$.
Then   $(u(t,x),v(t,x))=(\phi(x),\psi(x))$ is the unique solution of \eqref{1.1} with $u_0(x)=\phi(x)$.  By Lemma \ref{lem9.2} and $\inf_{x\in\mathbb R}{\phi(x)}>0$, we have that
\[
\left\|\phi(\cdot)-\frac{a}{b}\right\|_{\infty}+\left\|\psi(\cdot)-\frac{a\mu}{b\lambda}\right\|_{\infty}=0
\]
which implies that $(\phi,\psi)\equiv(\frac{a}{b},\frac{a\mu}{b\lambda})$.
\end{proof}

The following lemma states some properties of $u_0$ when it satisfies  \eqref{1.2} and \eqref{1.10}, which play crucial roles in the study of acceleration of propagation in the paper.
\begin{lemma}\label{lem9.0}
Assume that $u_0$ satisfies \eqref{1.2} and \eqref{1.10}. We have the following    statements:
\begin{itemize}
  \item[(i)] $u_0'(x)/u_0(x)\rightarrow 0$ as $x\rightarrow+\infty${\rm;}
  \item[(ii)] $u_0$ decays more slowly than any exponentially decaying function as $x\rightarrow+\infty$, namely,
      \[
      \forall~\kappa,~\exists~x_\kappa~\text{s.t.}~u_0(x)\geqslant e^{-\kappa x}~~\text{for all}~~x\in[x_\kappa,+\infty);
      \]
  \item[(iii)] for any $\gamma_1>0$, $\gamma_2>0$, and $\rho_1>\rho_2>0$, if the functions $y_1(\cdot)$ and $y_2(\cdot)$ satisfy
      \[
      u_0(y_1(t))=\gamma_1 e^{-\rho_1 t}\quad \text{and}\quad u_0(y_2(t))=\gamma_2 e^{-\rho_2 t}\quad \text{for}~t>0~\text{large enough},
      \]
      then we have that
      \[
      \lim_{t\rightarrow+\infty}(y_1(t)-y_2(t))=+\infty.
      \]
\end{itemize}
\end{lemma}
\begin{proof}
{(i)} Let $g(x)=u_0'(x)/u_0(x)\in C^1([\xi_0,+\infty))$.
Then we obtain $g'(x)+g^2(x)=u_0''(x)/u_0(x)$, which implies
\begin{equation}\label{9.1}
\lim\limits_{x\rightarrow+\infty} (g'(x)+g^2(x))=0.
\end{equation}
Note that $g(x)\leqslant 0$ for all $x\in[\xi_0,+\infty)$. In what follows, we prove that $\lim\limits_{x\rightarrow+\infty}g(x)$ exists in $(-\infty,0]\cup\{-\infty\}$.
Otherwise, there are two sequences $\{y_n\}_{n\in\mathbb N}$ and $\{z_n\}_{n\in\mathbb N}$  satisfying   that $y_n\rightarrow +\infty$, $z_n\rightarrow +\infty$ as $n\rightarrow\infty$, and
\[
\lim_{n\rightarrow\infty}g(y_n)=A,~\lim_{n\rightarrow\infty}g(z_n)=B~\text{with}~A,B\in (-\infty,0]\cup\{-\infty\}~\text{and}~A<B.
\]
Let $\varepsilon$ be a negative constant in $(A,B)$.
Since there exists $N_0>0$ such that $g(y_n)<\varepsilon$ and $g(z_n)>\varepsilon$ for any $n\geqslant N_0$, by the continuity of $g$ and
intermediate value theorem,
there exits $x_n\in (\min\{y_n,z_n\},\max\{y_n,z_n\})$ with $n\geqslant N_0$ such that $x_n\rightarrow+\infty$ as $n\rightarrow +\infty$, and
\begin{equation}\label{9.2}
g(x_n)=\varepsilon~\text{for any}~n\geqslant N_0.
\end{equation}
Then by \eqref{9.1}, we have that $\lim\limits_{n\rightarrow+\infty} g'(x_n)=-\varepsilon^2<0$, which implies    there is $N_1\geqslant N_0$ such that
\begin{equation}\label{9.3}
g'(x_n)<-\varepsilon^2/2<0~\text{for any}~n\geqslant N_1.
\end{equation}
Combining \eqref{9.2} with \eqref{9.3}, by the continuity of $g$, we can find a sequence   $\{\alpha_n\}$ with $n\geqslant N_1$ satisfying that
\[
x_n\leqslant \alpha_n\leqslant x_{n+1},~g(\alpha_n)<\varepsilon<0,~ g'(\alpha_n)>0~~\text{for any}~n\geqslant N_1.
\]
Then we have that   $g'(\alpha_n)+g^2(\alpha_n)>\varepsilon^2$ for any $n\geqslant N_1$, which is a contradiction with \eqref{9.1}. Therefore,   $\lim\limits_{x\rightarrow+\infty}g(x)$ exists in $(-\infty,0]\cup\{-\infty\}$.

To complete the proof of (i), we suppose by contradiction that $\lim\limits_{x\rightarrow+\infty}g(x)\neq0$.  Then by  \eqref{9.1} we have that
\[
\lim\limits_{x\rightarrow+\infty}(g'(x)+g^2(x))/ g^2(x)=0.
\]
Namely,  $\lim\limits_{x\rightarrow+\infty}g'(x)/ g^2(x)=-1$. For any fixed constant $y\in[\xi_0,+\infty)$, it holds that
\[
-\infty=\int_y^{+\infty} \frac{g'(x)}{g^2(x)} dx=\frac{1}{g(y)}-\frac{1}{g(+\infty)}
\]
This is a contradiction since we suppose   $g(+\infty)\neq0$. Therefore, we have that $\lim\limits_{x\rightarrow+\infty}g(x)=0$.

{(ii)} By Lemma \ref{lem9.0} (i),  we have that $u_0'(x)/u_0(x)\rightarrow 0$  as $x\rightarrow+\infty$.
For any  $\kappa>0$, we can  choose a constant $\eta$   in $(0,\kappa)$. Then there exists $x_\eta>0$ such that $u_0'(x)/u_0(x)\geqslant-\eta$ for $x\in[x_\eta,+\infty)$. By integrating this inequality from $x_\eta$ to $x$, we can get  that $u_0(x)/u_0(x_\eta)\geqslant e^{-\eta(x-x_\eta)}$ for any $x\in[x_\eta,+\infty)$.
Hence, there exists $C_\eta=u_0(x_\eta)e^{\eta x_\eta}>0$ such that
\[
u_0(x)\geqslant C_\eta e^{-\eta x}~\text{for all}~ x\in[x_\eta,+\infty).
\]
For any $\kappa$, we choose $x_\kappa=\max\{x_\eta, (\eta-\kappa)^{-1}\ln {C_\eta}\}$. It follows that
\[
u_0(x)\geqslant C_\eta e^{(\kappa-\eta) x}e^{-\kappa x}\geqslant e^{-\kappa x}~~\text{for all}~~x\in[x_\kappa,+\infty).
\]

{(iii)}
It suffices to prove that for any $M>0$, there exists $t_M>0$ such that
\[
y_1(t)-y_2(t)\geqslant M~\text{for all}~t\geqslant t_M.
\]
By \eqref{1.2}, \eqref{1.10}, and $\rho_1>\rho_2>0$, we can find a large constant $t_0\geqslant 0$ such that  the function $t\mapsto y_i(t)$ with $i\in\{1,2\}$ is  nondecreasing  on $[t_0,+\infty)$ and
\[
y_1(t)> y_2(t)~\text{for}~t\geqslant t_0,~y_2(t)\rightarrow +\infty~\text{as}~t\rightarrow +\infty.
\]
By $\rho_1>\rho_2>0$ and  Lemma \ref{lem9.0} (i), for any $M>0$,  there is a large constant $t_M\geqslant t_0$ such that
\begin{equation}\label{1.35}
\frac{\gamma_1}{\gamma_2} e^{(\rho_2-\rho_1)t}\leqslant\frac{1}{2}~\text{for all}~t\geqslant t_M,
\end{equation}
and $y_2(t_M)$   is large enough  such that
\[
y_2(t_M)\geqslant \xi_0,~\text{and}~u'_0(x)\geqslant-\frac{1}{2M} u_0(x)~\text{for all}~x\geqslant y_2( t_M).
\]
When $t\geqslant t_M$, it follow from $y_2(t)\geqslant y_2(t_M)$  that
\[
u'_0(x)\geqslant-\frac{1}{2M} u_0(x)~\text{for all}~x\geqslant y_2(t).
\]
Let $t_M$ be larger (if necessary) such  that
\[
u_0(y_1(t))=\gamma_1 e^{-\rho_1 t}~ \text{and}~ u_0(y_2(t))=\gamma_2 e^{-\rho_2 t}~\text{for all}~t\geqslant t_M.
\]
Then some calculations show that
\[
\gamma_1 e^{-\rho_1 t}-\gamma_2 e^{-\rho_2 t}=u_0(y_1(t))-u_0(y_2(t))=\int_{y_2(t)}^{y_1(t)}u_0'(x)dx\geqslant -\frac{1}{2M} \int_{y_2(t)}^{y_1(t)}u_0(x)dx,\quad t\geqslant t_M.
\]
Since $u_0$ is nonincreasing on $[\xi_0,+\infty)$ and $y_2(t)\geqslant y_2(t_M)\geqslant \xi_0$, we can get   $u_0(x)\leqslant u_0(y_2(t))$ for $x\geqslant y_2(t)$, which implies that
\[
\gamma_1 e^{-\rho_1 t}-\gamma_2 e^{-\rho_2 t}\geqslant -\frac{1}{2M} (y_1(t)-y_2(t))u_0(y_2(t)) =-\frac{1}{2M} (y_1(t)-y_2(t))\gamma_2 e^{-\rho_2 t},\quad t\geqslant t_M.
\]
By \eqref{1.35}, we have  that
\[
y_1(t)-y_2(t)\geqslant -2M \left(\frac{\gamma_1}{\gamma_2} e^{(\rho_2-\rho_1) t}-1\right)\geqslant M,\quad t\geqslant t_M.
\]
This completes the proof.
\end{proof}

\begin{remark}\em{
We remark that the proof of Lemma \ref{lem9.0} (i) was given in \cite{HamelRogues2010} but hard to understand. Here we provide a different and easier proof. Lemma \ref{lem9.0} (ii) was announced in \cite{HamelRogues2010} without proof. Hence here we supplement a proof.
}
\end{remark}

Let $L$ be a positive constant and  assume that $g: [0,L]\rightarrow \mathbb R$ is a Lipschitz continuous function.
For $\alpha\in\mathbb R$, we consider the equation
\begin{equation}\label{2.2}
\left\{
\begin{aligned}
&w_t=w_{xx}+ \alpha |w_x|+ g(w),~&&x\in\mathbb R,~t>0,\\
&w(x,0)=w_0(x),~&&x\in\mathbb R.
\end{aligned}
\right.
\end{equation}
The following lemma provides a comparison principle for \eqref{2.2}, which is a crucial tool used in the sequel.
\begin{lemma}[Comparison principle]\label{le2.1}
Suppose that the bounded nonnegative  functions $\bar w$ and  $\underline w$  are the upper and lower solutions of \eqref{2.2}, in the sense that
\[
\bar w_t-\bar w_{xx}- \alpha|\bar w_x|- g(\bar w)\geqslant0 \geqslant\underline  w_t-\underline  w_{xx}- \alpha|\underline  w_x|- g(\underline  w),\quad\text{for}~x\in\mathbb R,~t>0.
\]
If $\bar w(x,0)\geqslant \underline w(x,0)$ for all $x\in \mathbb R$, then $\bar w(x,t)\geqslant \underline w(x,t)$ for all   $(x,t)\in \mathbb R\times [0,+\infty)$.
\end{lemma}
\begin{proof}
Let $m(x,t)=\bar w(x,t)- \underline w(x,t)$ for $(x,t)\in \mathbb R\times [0,+\infty)$.  We have that $m(x,0)\geqslant 0$ for $x\in\mathbb R$ and $m_x(x,t)=\bar w_x(x,t)- \underline w_x(x,t)$. It follows that
\[
-|m_x(x,t)|\leqslant|\bar w_x(x,t)|-|\underline w_x(x,t)|\leqslant |m_x(x,t)|~\text{for}~(x,t)\in \mathbb R\times [0,+\infty).
\]
Let $G>0$ be a Lipschitz constant of $g$.
It follows that
\begin{equation}\label{2.4}
\begin{aligned}
m_t(x,t)&\geqslant m_{xx}(x,t)+\alpha|\bar w_x(x,t)|-\alpha|\underline w_x(x,t)|+g(\bar w(x,t))-g(\underline w(x,t))\\
&\geqslant m_{xx}(x,t)-|\alpha m_x(x,t)|-G|m(x,t)|\quad\text{for}~(x,t)\in \mathbb R\times [0,+\infty)
\end{aligned}
\end{equation}
Then it remains to prove that $m(x,t)\geqslant 0$ for all $(x,t)\in \mathbb R\times [0,+\infty)$.
Suppose, by contradiction, that there exists $t_0\in(0,+\infty)$ such that
\[
\inf_{x\in\mathbb R} \{m(x,t_0)\}<0.
\]

Let $p(\cdot):[0,1]\rightarrow[1,3]$ be a smooth increasing auxiliary  function satisfying that
\begin{equation}\label{2.99}
p(0)=1,~p(1)=3,~\text{and}~p'(0)=p''(0)=p'(1)=p''(1)=0.
\end{equation}
We denote
\begin{equation}\label{2.7}
C_1=\max_{[0,1]}\{|p'(x)|\}\quad\text{and}\quad C_2=\max_{[0,1]}\{|p''(x)|\}.
\end{equation}
Let $K$ be a large constant satisfying
\begin{equation}\label{2.9}
K>\frac{2}{3}C_2+\frac{2}{3}|\alpha| C_1+\frac{8}{3}G.
\end{equation}
Define
\[
h(t)=-e^{-K t}\inf_{x\in\mathbb R}\{m(x,t)\},\quad t\in [0,t_0].
\]
It follows that  $h(0)\leqslant0$ and $h(t_0)>0$. Let $T_1\in (0,t_0]$ satisfy
\[
h(T_1)=H_0\triangleq \max_{t\in[0,t_0]}\{h(t)\}>0.
\]
Then we get
\begin{equation}\label{2.5}
\inf\limits_{x\in\mathbb R} \{m(x,T_1)\}=-H_0 e^{K T_1},
\end{equation}
and
\begin{equation}\label{2.3}
m(x,t)\geqslant \inf_{x\in\mathbb R}\{m(x,t)\}=-h(t) e^{Kt}\geqslant -H_0 e^{Kt}~\text{for all}~x\in\mathbb R,~t\in[0,t_0].
\end{equation}

Now we consider the following two  cases
\begin{itemize}
  \item[Case A] There exists $x_0\in\mathbb R$ such that $m(x_0,T_1)=\inf\limits_{x\in\mathbb R} \{m(x,T_1)\}$.
  \item[Case B]   $\liminf\limits_{x\rightarrow+\infty}\{m(x,T_1)\}=\inf\limits_{x\in\mathbb R} \{m(x,T_1)\}$ or $\liminf\limits_{x\rightarrow-\infty}\{w(x,T_1)\}=\inf\limits_{x\in\mathbb R} \{m(x,T_1)\}$.
\end{itemize}
If case A holds, we can obtain that  $m(x_0,T_1)=-H_0 e^{K T_1}$, $m_x(x_0,T_1)=0$, and $m_{xx}(x_0,T_1)\geqslant0$. By \eqref{2.5} and \eqref{2.3}, we have that
\[
m(x_0,t)+H_0 e^{Kt}\geqslant 0=m(x_0,T_1)+H_0 e^{K T_1}~\text{for}~t\in[0,T_1].
\]
It follows that
\begin{equation}
0\geqslant \left.\frac{d}{d t}\left(m(x_0,t)+H_0 e^{Kt}\right)\right|_{t=T_1}=m_t(x_0,T_1)+KH_0e^{KT_1}
\end{equation}
Then we get by \eqref{2.9}  that
\[
m_t(x_0,T_1)-m_{xx}(x_0,T_1)+|\alpha m_x(x_0,T_1)|+G|m(x_0,T_1)|\leqslant (-K+G)H_0 e^{K T_1}<0,
\]
which  contradicts \eqref{2.4}.

If case B holds, then there exists $x_1\in\mathbb R$ (far away from $0$) such that $m(x_1, T_1)< -\frac{3}{4}H_0 e^{K T_1}$.
We define  a function $z(\cdot): \mathbb R\rightarrow [1,3]$  as follows
\[
z(x)=\left\{
\begin{aligned}
&1~~~&&\quad\text{for}~|x|\leqslant |x_1|,\\
&p(|x|-|x_1|)&&\quad\text{for}~|x_1|<|x|< |x_1|+1,\\
&3~~~&&\quad\text{for}~|x|\geqslant |x_1|+1.
\end{aligned}
\right.
\]
It follows from \eqref{2.99} that $z(\cdot)\in C^2(\mathbb R)$.
For $\sigma>0$, we define
\[
\rho_{\sigma}(x,t)=-\left(\frac{1}{2}+\sigma z(x)\right)H_0 e^{K t}
\]
and
\[
\sigma^*=\inf\{\sigma>0~|~\rho_{\sigma}(x,t)\leqslant m(x,t)~\text{for all}~x\in\mathbb R,~t\in[0,T_1]\}.
\]
Then $\rho_{\frac{1}{4}}(x_1,T_1)=-\frac{3}{4}H_0 e^{K T_1}>m(x_1, T_1)$.
By \eqref{2.3}, we have that  $\rho_{\frac{1}{2}}(x,t)\leqslant-H_0 e^{K t}\leqslant m(x,t)$  for all $x\in\mathbb R$ and $t\in[0,T_1]$. The monotonicity of $\sigma \mapsto \rho_\sigma$ implies that $\frac{1}{4}\leqslant\sigma^*\leqslant\frac{1}{2}$.    When $|x|\geqslant |x_1|+1$, by \eqref{2.3} we get that
\[
\rho_{\sigma^*}(x,t)=-\left(\frac{1}{2}+3 \sigma^*\right)H_0 e^{Kt}\leqslant -\frac{5}{4}H_0 e^{K t}< m(x,t)~\text{for all}~t\in [0,T_1].
\]
By the definition of $\sigma^*$, there exist $T_2\in(0,T_1]$ and $x_2\in[-|x_1|-1,|x_1|+1]$ such that
\begin{equation}\label{2.6}
m(x_2,T_2)=\rho_{\sigma^*}(x_2,T_2),~\text{and}~ m(x,t)\geqslant\rho_{\sigma^*}(x,t)~\text{for all}~x\in\mathbb R,~t\in[0,T_1],
\end{equation}
From $\frac{1}{4}\leqslant\sigma^*\leqslant\frac{1}{2}$ and $1\leqslant z(x)\leqslant 3$, it follows that
\begin{equation}\label{2.8}
-2H_0 e^{K T_2}\leqslant m(x_2,T_2)=\rho_{\sigma^*}(x_2,T_2)\leqslant-\frac{3}{4}H_0 e^{K T_2}.
\end{equation}
We have from \eqref{2.6} that
\[
\left.\frac{\partial}{\partial x} \Big(m(x,T_2)-\rho_{\sigma^*}(x,T_2)\Big)\right|_{x=x_2}=0,\quad\text{and}~\left.\frac{\partial^2}{\partial x^2} \Big(m(x,T_2)-\rho_{\sigma^*}(x,T_2)\Big)\right|_{x=x_2}\geqslant0.
\]
It follows from \eqref{2.7} and $\sigma^*\leqslant\frac{1}{2}$ that
\[
\begin{aligned}
&|m_{x}(x_2,T_2)|=\left|\frac{\partial}{\partial x}\rho_{\sigma^*}(x_2,T_2)\right|=\sigma^* |z'(x_2)|H_0 e^{K T_2}\leqslant\frac{1}{2} C_1 H_0 e^{K T_2},\\
&m_{xx}(x_2,T_2)\geqslant \frac{\partial^2}{\partial x^2}\rho_{\sigma^*}(x_2,T_2)=-\sigma^* z''(x_2)H_0 e^{K T_2}\geqslant -\frac{1}{2}C_2 H_0 e^{K T_2}.
\end{aligned}
\]
Moreover, it follows from \eqref{2.6} that
\[
\left.\frac{\partial}{\partial t}\left(m(x_2,t)-\rho_{\sigma^*}(x_2,t)\right)\right|_{t=T_2}\leqslant 0,
\]
which along  with \eqref{2.8} implies  that
\[
m_t(x_2,T_2)\leqslant\frac{\partial}{\partial t} \rho_{\sigma^*}(x_2,T_2)=K\rho_{\sigma^*}(x_2,T_2)\leqslant-\frac{3}{4}KH_0 e^{K T_2}.
\]
Therefore, we get from \eqref{2.9} that
\[
\begin{aligned}
&m_t(x_2,T_2)-m_{xx}(x_2,T_2)+|\alpha m_x(x_2,T_2)|+G|m(x_2,T_2)|\\
\leqslant &\left(-\frac{3}{4}K+\frac{1}{2}C_2+\frac{1}{2}|\alpha|C_1+2G\right)H_0 e^{K T_2}<0,
\end{aligned}
\]
which  contradicts \eqref{2.4}. Thus the proof is completed.
\end{proof}

\section{Acceleration of propagation}
In this section, we are devoted to the proof of Theorem  \ref{th1}.
By substituting the second equation of \eqref{1.1} into the first one, we get
\begin{equation}\label{1.3}
u_t=u_{xx}-\chi u_x v_x+ g(u,v),~x\in\mathbb R,~t>0,
\end{equation}
where
\begin{equation}\label{1.4}
g(u,v)=u[a-bu+\chi\mu u-\chi\lambda  v].
\end{equation}
The following lemma provides an upper bound of the moving level set
$E_\omega(t)$ for large $t$.
\begin{lemma}\label{lem3.1}
Under the same assumptions as in Theorem \ref{th1}, for any $\epsilon>0$  and $\gamma_1>0$, if $\zeta(t)$ satisfies that
\[
u_0(\zeta(t))=\gamma_1 e^{-(a+\epsilon)t}~\text{for}~t~\text{large enough},
\]
then for any $\omega\in(0,a/b)$, there is a constant $T_1\geqslant 0$ such that
\begin{equation}\label{2.1}
E_\omega(t)\subseteq (-\infty,\zeta(t)]~\text{for any}~t\geqslant T_1.
\end{equation}
\end{lemma}

\begin{proof}
By $u\in C([0,\infty), C_{\rm unif}^b(\mathbb R))$ and \eqref{1.12}, we can define
\begin{equation}\label{2.98}
L=\sup_{t\in[0,+\infty)}\|u(\cdot,t)\|_{\infty}<+\infty,
\end{equation}
which implies that
\begin{equation}\label{1.16}
0\leqslant u(x,t)\leqslant L~\text{for all}~(x,t)\in\mathbb R\times[0,+\infty).
\end{equation}
Note that for  any given function $u$, the second equation of  \eqref{1.1} can be regarded as an ordinary differential equation for $v$. As stated in \cite[Lemma 2.1]{SSX2019}, its solution in $C_\text{unif}^b(\mathbb R)$ is written as
\[
v(x,t)=\frac{\mu}{\lambda} J*u(x,t) =\frac{\mu}{2\sqrt{\lambda}}\int_{\mathbb R}e^{-\sqrt{\lambda}|x-z|}u(z,t)dz\quad\text{with}~ J(x)=\frac{\sqrt{\lambda}}{2}e^{-\sqrt{\lambda}|x|}.
\]
It follows from \eqref{1.16} and  $\int_{\mathbb R}J(x)dx=1$ that
\[
0\leqslant v(x,t)\leqslant\frac{\mu L}{\lambda}\quad\text{for all}~x \in\mathbb R,~t>0.
\]
By \cite[Lemma 2.2]{SSX2019}, we have
$|v_x(x,t)|\leqslant \sqrt{\lambda}v(x,t)$, which along with the above inequality yields
\begin{equation}\label{1.99}
|v_x(x,t)|\leqslant \frac{\mu L}{\sqrt{\lambda}}\quad\text{for all}~x \in\mathbb R,~t>0.
\end{equation}
With \eqref{1.4}, we have that $g(u,v)\leqslant u(a-(b-\chi\mu) u)$.  It follows from  \eqref{1.3} and \eqref{1.99} that
\[
u_t=u_{xx}+\chi v_x u_x+ g(u,v)\leqslant u_{xx}+\frac{\chi\mu L}{\sqrt{\lambda}} |u_x|+ u(a-(b-\chi\mu) u)\quad\text{for all}~x \in\mathbb R,~t>0.
\]
Then $u(x,t)$ is a lower solution of the following  equation
\begin{equation}\label{1.14}
w_t=w_{xx}+\frac{\chi\mu L}{\sqrt{\lambda}} |w_x|+ w(a-(b-\chi\mu) w),\quad~x \in\mathbb R,~t>0.
\end{equation}

By \eqref{1.2}, \eqref{1.10}, and Lemma \ref{lem9.0} (i),
we have that
\begin{equation}\label{1.23}
u_0(x)\rightarrow 0,~u_0'(x)/ u_0(x)\rightarrow 0,~\text{and}~u_0''(x)/ u_0(x)\rightarrow 0~\text{as}~x\rightarrow +\infty.
\end{equation}
Let $\xi_0$ be the constant  in \eqref{1.10}. Then it follows from  \eqref{1.2} that  $\inf_{x\in(-\infty,\xi_0]}\{u_0(x)\}>0$. Then
for any $\epsilon>0$, we can  choose a constant  $\xi_1\in[\xi_0,+\infty)$ such that
\begin{equation}\label{1.19}
|u_0''(x)|\leqslant \frac{\epsilon}{4} u_0(x),\quad\frac{\chi\mu L}{\sqrt{\lambda}} |u_0'(x)|\leqslant \frac{\epsilon}{4} u_0(x)\quad \text{for all}~x\geqslant \xi_1,
\end{equation}
and
\begin{equation}\label{1.20}
0<u_0(\xi_1)< \inf_{x\in(-\infty,\xi_0]}\{u_0(x)\}.
\end{equation}
We define a   function $\theta(\cdot):[0,+\infty)\rightarrow (0,+\infty)$ satisfying
\[\left\{
\begin{aligned}
&\theta'(t)=\theta(t)(a-(b-\chi\mu) \theta(t)),\quad t>0,\\
&\theta(0)=\theta_0,
\end{aligned}
\right.
\]
where
\[
\theta_0=\max\left\{\sup_{x\in\mathbb R}\{u_0(x)\}, \frac{a}{b-\chi\mu}\right\}.
\]
Note that $\theta(\cdot)$ is nonincreasing, and if $\sup_{x\in\mathbb R}\{u_0(x)\}\leqslant a/(b-\chi\mu)$, then $\theta(t)\equiv a/(b-\chi\mu)$.
Define
\begin{equation}\label{2.13}
\bar w(x,t)= \min\left\{C u_0(x)e^{\rho t},\theta(t)\right\}~ ~\text{for}~x\in\mathbb R,~t\geqslant0,
\end{equation}
where
\[
\rho=a+\epsilon/2\quad \text{and}\quad C=\frac{\theta_0}{u_0(\xi_1)}\geqslant \frac{\sup_{x\in\mathbb R}\{u_0(x)\}}{u_0(\xi_1)}\geqslant1.
\]
Now we show that $\bar w(x,t)$ is an upper solution of \eqref{1.14}.
When $\bar w(x,t)=\theta(t)$, we easily check
\[
\bar w_t-\bar w_{xx}-\frac{\chi\mu L}{\sqrt{\lambda}} |\bar w_x|- \bar w(a-(b-\chi\mu) \bar w)=\theta'(t)-\theta(t)(a-(b-\chi\mu) \theta(t))=0.
\]
When $\bar w(x,t)=C u_0(x)e^{\rho t}$, we have that
\[
C u_0(x)\leqslant C u_0(x)e^{\rho t}\leqslant \theta(t)\leqslant\theta_0,
\]
which implies   that $u_0(x)\leqslant\theta_0/C=u_0(\xi_1)$. By  \eqref{1.20} and $u_0'\leqslant0$ on $[\xi_0,+\infty)$, we obtain  $x\geqslant \xi_1$.
From  \eqref{1.19} and $\bar w(x,t)=C u_0(x)e^{\rho t}$, it follows that
\[
\begin{aligned}
&\bar w_t-\bar w_{xx}-\frac{\chi\mu L}{\sqrt{\lambda}} |\bar w_x|- \bar w(a-(b-\chi\mu) \bar w)\\
\geqslant &\rho C u_0(x)e^{\rho t}-C u_0''(x)e^{\rho t}-\frac{\chi\mu L}{\sqrt{\lambda}}C |u_0'(x)|e^{\rho t}-a C u_0(x)e^{\rho t}\\
\geqslant &C u_0(x)e^{\rho t}\left(\rho-\frac{\epsilon}{4}-\frac{\epsilon}{4}-a\right)=0.
\end{aligned}
\]
Therefore, $\bar w(x,t)$ is an upper solution of \eqref{1.14}.

By $C\geqslant1$ and $\theta_0\geqslant \sup_{x\in\mathbb R}\{u_0(x)\}$, we get that
\[
\bar w(x,0)= \min\left\{C u_0(x),\theta_0\right\}\geqslant u_0(x)~ ~\text{for all}~x\in\mathbb R.
\]
Applying Lemma \ref{le2.1} to \eqref{1.14}, we have that
\begin{equation}\label{1.22}
u(x,t)\leqslant \bar w(x,t) ~\text{for all}~x\in\mathbb R,~t\geqslant0.
\end{equation}
For $\omega\in(0,a/b)$ and $t\geqslant 0$, when $E_{\omega}(t)$ is an empty set,   \eqref{2.1} is obvious.
When $E_{\omega}(t)$ is nonempty, for any $x_\omega(t)\in E_\omega(t)$, we get from \eqref{1.22} that
\begin{equation}\label{2.12}
\omega=u(x_\omega(t),t)\leqslant \bar w (x_\omega(t),t)\leqslant C u_0(x_\omega(t))e^{\rho t}.
\end{equation}
Since  $\zeta(\cdot)$ satisfies
\[
u_0(\zeta(t))=\gamma_1 e^{-(a+\epsilon)t}~\text{for}~t~\text{large enough},
\]
we can get from \eqref{1.2} and \eqref{1.10}  that $\zeta(t)\rightarrow +\infty$ as $t\rightarrow +\infty$.
Let $T_1$ be a  constant large enough such that
$\zeta(t)\geqslant \xi_0$, $u_0(\zeta(t))=\gamma_1 e^{-(a+\epsilon)t}$ for all $t\geqslant T_1$, and $T_1> \max\{0,2 \epsilon^{-1}\ln \left(\gamma_1 C/\omega\right)\}$.
By \eqref{2.12} and $\rho=a+\epsilon/2$, we get  that
\[
 u_0(x_\omega(t))\geqslant C^{-1} \omega e^{-\rho t}\geqslant C^{-1}\omega  e^{\epsilon T_1/2}e^{-(a+\epsilon) t}> \gamma_1e^{-(a+\epsilon) t}=u_0(\zeta(t))\quad\text{for all}~t\geqslant T_1.
\]
This implies  $x_\omega(t)\leqslant\zeta(t)$ for all $t\geqslant T_1$ from \eqref{1.10}.
Hence, we obtain  \eqref{2.1} and complete the proof.
\end{proof}

In the following lemma, we give the lower bound of the moving level set
$E_\omega(t)$ for large $t$.
\begin{lemma}\label{lem3.2}
Under the same assumptions as in Theorem \ref{th1},
for any $\epsilon\in(0,a)$ and $\gamma_2>0$, if $\eta(t)$   satisfies
\[
u_0(\eta(t))=\gamma_2 e^{-(a-\epsilon)t}  ~\quad\text{for}~t~\text{large enough},
\]
then for any $\omega\in(0,a/b)$, there is a constant $T_2\geqslant0$ such that
\begin{equation}\label{2.96}
E_\omega(t)\subseteq [\eta(t),+\infty)~\text{for any}~t\geqslant T_2.
\end{equation}
\end{lemma}
\begin{proof}
By Lemma 2.5 and (2.21) in \cite{SSX2019}, for any $r\gg 1$ and $p>1$, there exist $C_{r,p}>0$ and $\varepsilon_r>0$ such that
\begin{equation}\label{1.9}
\chi\lambda v(x,t)\leqslant C_{r,p}(u(x,t))^{\frac{1}{p}}+q\varepsilon_r \quad \quad\text{for all}~x \in\mathbb R,~t>0,
\end{equation}
where $q\triangleq \max\{\|u_0\|_{\infty}, a/(b-\chi\mu)\}$ and  the constant $\varepsilon_r$ satisfies  $\lim_{r\rightarrow\infty}\varepsilon_r =0$. Fix $\epsilon\in (0,a)$ and choose $r$ large enough such that
\[
q\varepsilon_r<\epsilon/2.
\]
Denote
\[
\delta= 1/p\quad\text{and}\quad M=C_{r,p}+(b-\chi\mu)L^{1-\delta},
\]
where $L$ is defined by \eqref{2.98}.
From  \eqref{1.4}, \eqref{1.16}, and \eqref{1.9}, it follows that
\[
\begin{aligned}
g(u,v)&=u[a-(b-\chi\mu) u-\chi\lambda  v]\\
&\geqslant u[a-(b-\chi\mu) L^{1-\delta}u^{\delta}-C_{r,p}u^{\delta}-q\varepsilon_r]\\
&\geqslant (a-\epsilon/2)u-Mu^{1+\delta}.
\end{aligned}
\]
We get from \eqref{1.3} and \eqref{1.99} that
\[
u_t=u_{xx}-\chi v_x u_x+ g(u,v)\geqslant u_{xx}-\frac{\chi\mu L}{\sqrt{\lambda}} |u_x|+ (a-\epsilon/2)u-Mu^{1+\delta}~\quad\text{for all}~x \in\mathbb R,~t>0.
\]
Then $u(x,t)$ is an upper solution of the following  equation
\begin{equation}\label{1.17}
w_t=w_{xx}-\frac{\chi\mu L}{\sqrt{\lambda}} |w_x|+ (a-\epsilon/2)w-Mw^{1+\delta},\quad ~x \in\mathbb R,~t>0.
\end{equation}

Choose a constant  $\rho>0$ satisfying that
\[
\max\left\{a-\epsilon,\frac{a-\epsilon/2}{1+\delta}\right\}<\rho<a-\epsilon/2.
\]
By \eqref{1.23} and \eqref{1.2},
we can find a constant  $\xi_2\in[\xi_0,+\infty)$ large enough such that
\begin{eqnarray}
G_1(x)&&\triangleq\frac{|u_0''(x)|}{u_0(x)}+\frac{\chi \mu L}{\sqrt{\lambda}}\frac{|u_0'(x)|}{u_0(x)}\leqslant (a-\epsilon/2)-\rho,\label{1.24}\\
G_2(x)&&\triangleq (1+\delta)\left[\frac{|u_0''(x)|}{u_0(x)}+\delta\Big(\frac{u_0'(x)}{u_0(x)}\Big)^2\right]+\frac{\chi\mu L}{\sqrt{\lambda}}(1+\delta)\frac{|u_0'(x)|}{u_0(x)} \nonumber \\
&&\leqslant\frac{\rho(1+\delta)-(a-\epsilon/2)}{2}\label{1.25}
\end{eqnarray}
for all $x\geqslant \xi_2$,
and
\begin{equation}\label{1.28}
0<u_0(\xi_2)< \inf_{x\in(-\infty,\xi_0]}\{u_0(x)\}.
\end{equation}
We define a function
\[
F(s)=s-Bs^{1+\delta},~s\geqslant 0,
\]
where
\begin{equation}\label{1.29}
B=\max\left\{\frac{1}{u_0^{\delta}(\xi_2)},~\frac{2M}{\rho(1+\delta)-(a-\epsilon/2)}\right\}.
\end{equation}
Denote $s_0\triangleq B^{-1/\delta}\leqslant u_0(\xi_2)$ and $s_1\triangleq(1+\delta)^{-1/\delta}B^{-1/\delta}<s_0$. Then  we have that
\begin{equation}\label{1.26}
F(s)\leqslant 0~ \text{for all}~s\geqslant s_0,
\end{equation}
and
\begin{equation}\label{1.27}
F(s)\leqslant F_0\triangleq\max_{s\geqslant0}\{F(s)\}= F(s_1)=\frac{\delta B^{-1/\delta}}{(1+\delta)^{1+1/\delta}}\leqslant B^{-1/\delta} \leqslant u_0(\xi_2)~\text{for all}~s\geqslant 0.
\end{equation}

Next we prove that  the function defined by
\[
\underline w (x,t)=\max\left\{F(u_0(x)e^{\rho t}),0\right\}
=\max\left\{u_0(x)e^{\rho t}-Bu_0^{1+\delta}(x)e^{\rho (1+\delta)t},0\right\},~x\in\mathbb R,~t\geqslant 0
\]
is a lower solution of \eqref{1.17}.
It suffices to check that $\underline w $ is a lower solution in the region where $\underline w>0$.
From \eqref{1.26}, it follows that $u_0(x)e^{\rho t}< s_0 \leqslant u_0(\xi_2)$, which implies   $u_0(x)<u_0(\xi_2)$.
By \eqref{1.28} and $u_0'(x)\leqslant0$ on $[\xi_0,+\infty)$, we have that
\begin{equation}\label{1.30}
x>\xi_2,~\text{when}~\underline w(x,t)>0.
\end{equation}
Note that
\begin{eqnarray*}
&&\left|\underline w_x(x,t)\right|\leqslant H_1(x,t)\triangleq |u_0'(x)|e^{\rho t}+B(1+\delta)u_0^\delta(x)|u_0'(x)|e^{\rho(1+\delta) t},\\
&&\left|\underline w_{xx}(x,t)\right|\leqslant H_2(x,t)\triangleq |u_0''(x)|e^{\rho t}+ B(1+\delta)\left[u_0^\delta(x)|u_0''(x)|+\delta u_0^{\delta-1}(x)(u_0'(x))^2\right] e^{\rho(1+\delta)t}.
\end{eqnarray*}
Some calculations show that
\[
\begin{aligned}
&\underline w_t(x,t)-\underline w_{xx}(x,t)+\frac{\chi\mu L}{\sqrt{\lambda}} \left|\underline w_x(x,t)\right|- (a-\epsilon/2)\underline w(x,t)+M\underline w^{1+\delta}(x,t)\\
\leqslant & \rho u_0(x)e^{\rho t}-B \rho (1+\delta)u_0^{1+\delta}(x)e^{\rho (1+\delta)t}+H_2(x,t)+\frac{\chi\mu L}{\sqrt{\lambda}}  H_1(x,t)\\
&-(a-\epsilon/2)u_0(x)e^{\rho t}+(a-\epsilon/2)Bu_0^{1+\delta}(x)e^{\rho (1+\delta)t} +Mu_0^{1+\delta}(x)e^{\rho (1+\delta)t}\\
=&u_0(x)e^{\rho t}\left[\rho+G_1(x)-(a-\epsilon/2)\right]+Bu_0^{1+\delta}(x)e^{\rho (1+\delta)t}\left[- \rho (1+\delta)+G_2(x)+(a-\epsilon/2)+M/B\right].
\end{aligned}
\]
By \eqref{1.24}, \eqref{1.25}, and \eqref{1.29}, we obtain that
\[
\rho+G_1(x)-(a-\epsilon/2)\leqslant 0,~\text{and}~-\rho (1+\delta)+G_2(x)+(a-\epsilon/2)+M/B\leqslant 0.
\]
Then we have that
\[
\underline w_t(x,t)-\underline w_{xx}(x,t)+\frac{\chi\mu L}{\sqrt{\lambda}} \left|\underline w_x(x,t)\right|- (a-\epsilon/2)\underline w(x,t)+M\underline w^{1+\delta}(x,t)\leqslant 0.
\]
Therefore, $\underline w (x,t)$ is a lower solution of \eqref{1.17}.

Note that $s_1 e^{-\rho t}\leqslant s_1<u_0(\xi_2)$ for all $t\geqslant0$. By \eqref{1.2} and the nonincreasing property of $u_0(\cdot)$ on $[\xi_2,+\infty)$, there exists $z(t)\geqslant \xi_2$
for any $t\geqslant0$  satisfying $\lim_{t\rightarrow\infty}z(t)=+\infty$ and
\begin{equation}\label{1.36}
u_0(z(t))=s_1 e^{-\rho t},~t\geqslant0.
\end{equation}
It follows that
\begin{equation}\label{2.95}
\underline w(z(t),t)=\max\left\{F(s_1),0\right\}=F_0~\text{for all}~t\geqslant0.
\end{equation}
Next we prove that
\begin{equation}\label{1.31}
u(x,t)\geqslant F_0~~\text{for all}~x\leqslant z(t),~t\geqslant0.
\end{equation}
Consider the following function
\[
w_0(x)=
\left\{
\begin{aligned}
&u_0(\xi_2),~&&\text{for}~x\leqslant \xi_2,\\
&u_0(x),~~&&\text{for}~x\geqslant \xi_2.
\end{aligned}
\right.
\]
Note that $w_0(\cdot)$ is a nonincreasing function, namely $w_0(x+y)\leqslant w_0(x)$ for any $x\in\mathbb R$ and $y\geqslant 0$.
By \eqref{1.28} and the nonincreasing property of $u_0(\cdot)$ on $[\xi_0,+\infty)$, we get $w_0(x)\leqslant u_0(x)$ for $x\in\mathbb R$.
With \eqref{1.27}, we can easily check that
\[
\underline w (x,0)=\max\{u_0(x)-Bu_0^{1+\delta}(x),0\}\leqslant \min\{u_0(x),F_0\}\leqslant \min\{u_0(x),u_0(\xi_2)\}=w_0(x),~x\in\mathbb R.
\]
Then for any constant $y\geqslant0$, we have that
\[
\underline w (x+y,0)\leqslant w_0(x+y)\leqslant w_0(x)\leqslant u_0(x),~x\in\mathbb R.
\]
Since   $\underline w (x,t)$ is a lower solution of \eqref{1.17}, for any constant $y\geqslant0$, $\underline w(x+y,t)$ is also a lower solution of \eqref{1.17}. Note that $u(x,t)$ is an upper solution of \eqref{1.17}. Applying Lemma \ref{le2.1} to \eqref{1.17}, we have that
\[
u(x,t)\geqslant\underline w(x+y,t)~\text{for all}~x\in\mathbb R,~t\geqslant0,~y\geqslant 0.
\]
For any $t\geqslant0$ and $x\leqslant z(t)$, we can choose $y=z(t)-x\geqslant 0$, and it holds by \eqref{2.95} that
\[
u(x,t)\geqslant\underline w(x+y,t)=\underline w(z(t),t)=F_0.
\]
Then we obtain \eqref{1.31}, which implies that
\begin{equation}\label{1.32}
\liminf_{t\rightarrow+\infty}\inf_{x\leqslant z(t)} u(x,t)\geqslant F_0.
\end{equation}

Recall that for any $\epsilon\in(0,a)$ and $\gamma_2>0$, the function  $\eta(\cdot)$ satisfies
\begin{equation}\label{1.37}
u_0(\eta(t))=\gamma_2 e^{-(a-\epsilon)t}\quad\text{for}~t~\text{large enough}.
\end{equation}
Next  we prove that
\begin{equation}\label{1.33}
\lim_{t\rightarrow+\infty}\sup_{x\leqslant \eta(t)} \left|u(x,t)-\frac{a}{b}\right|=0.
\end{equation}
Suppose by contradiction that \eqref{1.33} does not hold. Then there exist two sequences $(t_n)_{n\in\mathbb N}$ and $(x_n)_{n\in\mathbb N}$ satisfying $t_n\rightarrow+\infty$, $x_n\leqslant \eta(t_n)$ for $n$ large enough,
and
\begin{equation}\label{1.34}
\liminf_{n\rightarrow\infty}\left|u(x_n,t_n)-\frac{a}{b}\right|>0.
\end{equation}
We consider the function sequence
$(u^n(x,t),v^n(x,t))=(u(x+x_n,t+t_n),v(x+x_n,t+t_n))$.
By some  estimates similar to those in parabolic equations, we have that
$(u_n(x,t),v_n(x,t))$ converges, up to extraction of a subsequence, to some function $(u^*(x,t),v^*(x,t))$ locally uniformly in $C^{2,1}(\mathbb R\times\mathbb R)\times C^{2,1}(\mathbb R\times\mathbb R)$.  It follows that   $(u^*(x,t),v^*(x,t))$ is an entire solution (which means it is  defined for all $x\in\mathbb R$ and $t\in\mathbb R$) of the following system
\begin{equation}\label{1.39}
\left\{
\begin{aligned}
&u_t= u_{xx}-\chi (uv_x)_x+u(a-bu),&&x\in\mathbb R,~t\in\mathbb R,\\
&0=v_{xx}-\lambda v+\mu u,&&x\in\mathbb R,~t\in\mathbb R.\\
\end{aligned}
\right.
\end{equation}
We claim that
\begin{equation}\label{2.97}
u^*(x,t)\geqslant F_0\quad \text{for all}~~x\in\mathbb R,~t\in\mathbb R.
\end{equation}
By  $x_n\leqslant \eta(t_n)$ and $u^*(x,t)=\lim_{n\rightarrow \infty} u(x+x_n,t+t_n)$, it holds that
\begin{equation}\label{1.38}
u^*(x,t)
\geqslant \liminf_{\tau\rightarrow +\infty} \inf_{y\leqslant \eta(\tau)} u(x+y,t+\tau)
=\liminf_{\tau\rightarrow +\infty}\inf_{y\leqslant x+ \eta(\tau-t)} u(y,\tau),~~x\in\mathbb R,~t\in\mathbb R.
\end{equation}
By \eqref{1.36} and \eqref{1.37}, we have that
\[
u_0(z(\tau))=s_1 e^{-\rho \tau},~\text{and}~u_0(\eta(\tau-t))=\gamma_2 e^{(a-\epsilon)t }e^{-(a-\epsilon)\tau }\quad\text{for}~\tau~\text{large enough}.
\]
For any fixed $t\in\mathbb R$, by $\rho >a-\epsilon>0$ and Lemma \ref{lem9.0} (iii), we can get that
\[
\lim_{\tau\rightarrow+\infty} z(\tau)-\eta(\tau-t)=+\infty.
\]
Then for any fixed $x\in\mathbb R$ and $t\in\mathbb R$, we have that
\[
z(\tau)\geqslant x+\eta(\tau-t)\quad\text{for}~\tau~\text{large enough}.
\]
It follows from \eqref{1.32} and \eqref{1.38} that
\[
u^*(x,t)\geqslant \liminf_{\tau\rightarrow +\infty}\inf_{y\leqslant z(\tau)} u(y,\tau)\geqslant F_0~\text{for all}~~x\in\mathbb R,~t\in\mathbb R,
\]
which means \eqref{2.97} holds.
Note that both $(u^*(x,-\infty),v^*(x,-\infty))$ and  $(u^*(x,+\infty),v^*(x,+\infty))$ are  the  solutions  of \eqref{2.11} satisfying
\[
\inf_{x\in\mathbb R}u^*(x,-\infty)\geqslant F_0>0,~\text{and}~\inf_{x\in\mathbb R}u^*(x,+\infty)\geqslant F_0>0.
\]
When $b>2\chi\mu$, Lemma \ref{lem9.1} implies  that $(u^*(\cdot,\pm\infty),v^*(\cdot,\pm\infty))\equiv (\frac{a}{b},\frac{a\mu}{b\lambda})$.
Then by the stability of the positive constant equilibrium $(\frac{a}{b},\frac{a\mu}{b\lambda})$, it holds  that $u^*(x,t)\equiv a/b$. In particular, we have
\[
\frac{a}{b}=u^*(0,0)=\lim_{n\rightarrow\infty}u(x_n,t_n),
\]
which contradicts \eqref{1.34}. Therefore, we obtain  \eqref{1.33}.

Finally, we complete the proof of Lemma \ref{lem3.2} by \eqref{1.33}, which implies that for any $\omega \in(0,a/b)$, there exists $T_2\geqslant 0$ such that
\begin{equation}\label{2.94}
\inf_{x\leqslant \eta(t)}u(x,t)>\omega~\text{for all}~t\geqslant T_2.
\end{equation}
For any $\omega \in(0,a/b)$, when $E_{\omega}(t)$ is an empty set,   \eqref{2.96} is obvious.
When $E_{\omega}(t)$   is nonempty, for any $x_\omega(t)\in E_\omega(t)$, we get from  \eqref{2.94} that
\[
x_\omega(t)\geqslant \eta(t)~\text{for all}~t\geqslant T_2,
\]
which implies \eqref{2.96} and completes the proof.

\end{proof}

\begin{proof}[\em{\bf Proof of Theorem \ref{th1}}]
For any fixed  $t\geqslant0$, it follows from \eqref{2.13} and \eqref{1.22} that
\[
\lim_{x\rightarrow+\infty} u(x,t)\leqslant \lim_{x\rightarrow+\infty} \bar w(x,t)\leqslant \lim_{x\rightarrow+\infty} C u_0(x)e^{\rho t} =0
\]
which implies $\lim\limits_{x\rightarrow+\infty} u(x,t)=0$ due to the nonnegativity of $u$. By \eqref{1.33} and $\eta(t)\rightarrow+\infty$ as $t\rightarrow+\infty$, we have that
\[
\liminf_{x\rightarrow -\infty} u(x,t)\rightarrow a/b~\text{as}~t\rightarrow+\infty.
\]
Then for any $\omega\in(0,a/b)$, we can find $T_\omega\geqslant 0$ such that $\liminf\limits_{x\rightarrow -\infty} u(x,t)>\omega$ for all $t\geqslant T_\omega$.
By the continuity of $u(\cdot,t)$, the set $E_{\omega}(t)$ with $t\geqslant T_{\omega}$ is compact and nonempty. Then the assertion of Theorem \ref{th1} (i) is proved.

The result asserted in Theorem \ref{th1} (ii) is directly obtained from Lemmas \ref{lem3.1} and \ref{lem3.2} by denoting  $T=\max\{T_1,T_2\}$.

Next we prove Theorem \ref{th1} (iii) by  Lemma \ref{lem9.0} (ii) and Lemma \ref{lem3.2}.
It suffices to show that for any $N>0$, there exists a constant $T_N\geqslant0$ such that
\begin{equation}\label{2.93}
\frac{\inf\{E_\omega(t)\}}{t}\geqslant N~\text{for all}~t\geqslant T_N.
\end{equation}
Let  $\gamma_2>1$ and choose $T_0>0$ large enough such that
\[
u_0(\eta(t))=\gamma_2 e^{-(a-\epsilon)t}\quad\text{for all}~t\geqslant T_0.
\]
For any $N>0$,  we define  a constant  $\kappa$ by
\[
\kappa= \frac{a-\epsilon- R\ln \gamma_2}{N}
\]
where $R>0$ is a constant small enough such that $\kappa>0$.
Let  $x_\kappa$ denote the constant   given by  Lemma \ref{lem9.0} (ii). Recall that $\eta(t)\rightarrow +\infty$ as $t\rightarrow +\infty$.
Then there exists a constant $T_N\geqslant \max\{T_0,T_2,1/R\}$ such that
\[
\eta(t)\geqslant \max\{x_\kappa,\xi_0\}\quad \text{for all}~t\geqslant T_N.
\]
When $t\geqslant T_N\geqslant T_2$, for any $x_\omega(t)\in E_\omega(t)$, it follows from Lemma \ref{lem3.2} that
\[
\max\{x_\kappa,\xi_0\}\leqslant \eta(t)\leqslant x_\omega(t),
\]
which, along  with the nonincreasing property of $u_0$ on $[\xi_0,+\infty)$, implies that
\[ u_0(x_\omega(t))\leqslant u_0(\eta(t))=\gamma_2 e^{-(a-\epsilon)t},~t\geqslant T_N.
\]
By $x_\omega(t)\geqslant\eta(t)\geqslant x_\kappa$ for $t\geqslant T_N$, Lemma \ref{lem9.0} (ii) shows  that
\[
u_0(x_\omega(t))\geqslant e^{-\kappa x_\omega(t)},~t\geqslant T_N.
\]
Thus we have $e^{-\kappa x_\omega(t)}\leqslant\gamma_2e^{-(a-\epsilon)t}$, which implies by $T_N\geqslant 1/R$ that
\[
\frac{x_\omega(t)}{t}\geqslant\frac{a- \epsilon-t^{-1}\ln \gamma_2}{\kappa}\geqslant\frac{a- \epsilon-R\ln \gamma_2}{\kappa}=N,~t\geqslant T_N.
\]
Then we obtain \eqref{2.93}, which completes the proof.
\end{proof}

\section*{Acknowledgements}
The research of Z.-A. Wang is partially supported by a grant from the Research Grants Council of the Hong Kong Special Administration Region of China (No. PolyU 153055/18P  for GRF project) and an internal grant No. UAH0 (Project ID P0031504) from the Hong Kong Polytechnic University. The research of W.-B. Xu was partially supported by China Postdoctoral Science Foundation funded
project (2020T130679) and by the CAS AMSS-POLYU Joint Laboratory of Applied Mathematics
postdoctoral fellowship scheme.

\bibliography{references}

\begin{thebibliography}{10}

\bibitem{Adler66}
{\sc J.~Adler}, {\em Chemotaxis in bacteria}, Science, 153 (1966),
  pp.~708--716.

\bibitem{Adler69}
{\sc J.~Adler}, {\em Chemoreceptors in bacteria}, Science, 166 (1969),
  pp.~1588--1597.

\bibitem{Alfaro2017}
{\sc M.~Alfaro}, {\em Slowing {A}llee effect versus accelerating heavy tails in
  monostable reaction diffusion equations}, Nonlinearity, 30 (2017),
  pp.~687--702.

\bibitem{AlfaroCoville2017}
{\sc M.~Alfaro and J.~Coville}, {\em Propagation phenomena in monostable
  integro-differential equations: acceleration or not?}, J. Differential
  Equations, 263 (2017), pp.~5727--5758.

\bibitem{AW1975}
{\sc D.~G. Aronson and H.~F. Weinberger}, {\em Nonlinear diffusion in
  population genetics, combustion, and nerve pulse propagation}, in Partial
  differential equations and related topics ({P}rogram, {T}ulane {U}niv., {N}ew
  {O}rleans, {L}a., 1974), 1975, pp.~5--49. Lecture Notes in Math., Vol. 446.

\bibitem{AW1978}
{\sc D.~G. Aronson and H.~F. Weinberger}, {\em Multidimensional nonlinear
  diffusion arising in population genetics}, Adv. in Math., 30 (1978),
  pp.~33--76.

\bibitem{BootyHabermanMinzoni1993}
{\sc M.~R. Booty, R.~Haberman, and A.~A. Minzoni}, {\em The accommodation of
  traveling waves of {F}isher's type to the dynamics of the leading tail}, SIAM
  J. Appl. Math., 53 (1993), pp.~1009--1025.

\bibitem{budrene1995dynamics}
{\sc E.~Budrene and H.~Berg}, {\em Dynamics of formation of symmetrical
  patterns by chemotactic bacteria}, Nature, 376 (1995), pp.~49--53.

\bibitem{CabreRoquejoffre2013}
{\sc X.~Cabr\'{e} and J.-M. Roquejoffre}, {\em The influence of fractional
  diffusion in {F}isher-{KPP} equations}, Comm. Math. Phys., 320 (2013),
  pp.~679--722.

\bibitem{Chae}
{\sc M.~Chae, K.~Choi, K.~Kang, and J.~Lee}, {\em Stability of planar traveling
  waves in a {K}eller-{S}egel equation on an infinite strip domain}, {J.
  Differential Equations}, 265 (2018), pp.~237--279.

\bibitem{CoulonRoquejoffre2012}
{\sc A.-C. Coulon and J.-M. Roquejoffre}, {\em Transition between linear and
  exponential propagation in {F}isher-{KPP} type reaction-diffusion equations},
  Comm. Partial Differential Equations, 37 (2012), pp.~2029--2049.

\bibitem{Davis}
{\sc P.~Davis, P.~van Heijster, and R.~Marangell}, {\em Absolute instabilities
  of travelling wave solutions in a {K}eller-{S}egel model}, Nonlinearity, 30
  (2017), pp.~4029--4061.

\bibitem{FelmerYangari2013}
{\sc P.~Felmer and M.~Yangari}, {\em Fast propagation for fractional {KPP}
  equations with slowly decaying initial conditions}, SIAM J. Math. Anal., 45
  (2013), pp.~662--678.

\bibitem{FinkelshteinTkachov2019}
{\sc D.~Finkelshtein and P.~Tkachov}, {\em Accelerated nonlocal nonsymmetric
  dispersion for monostable equations on the real line}, Appl. Anal., 98
  (2019), pp.~756--780.

\bibitem{Garnier2011}
{\sc J.~Garnier}, {\em Accelerating solutions in integro-differential
  equations}, SIAM J. Math. Anal., 43 (2011), pp.~1955--1974.

\bibitem{HamelHenderson2020}
{\sc F.~Hamel and C.~Henderson}, {\em Propagation in a {F}isher-{KPP} equation
  with non-local advection}, J. Funct. Anal., 278 (2020), pp.~108426, 53.

\bibitem{HamelGadin2012}
{\sc F.~Hamel and G.~Nadin}, {\em Spreading properties and complex dynamics for
  monostable reaction-diffusion equations}, Comm. Partial Differential
  Equations, 37 (2012), pp.~511--537.

\bibitem{HamelRogues2010}
{\sc F.~Hamel and L.~Roques}, {\em Fast propagation for {KPP} equations with
  slowly decaying initial conditions}, J. Differential Equations, 249 (2010),
  pp.~1726--1745.

\bibitem{HeWuWu2017}
{\sc J.~He, Y.~Wu, and Y.~Wu}, {\em Large time behavior of solutions for
  degenerate {$p$}-degree {F}isher equation with algebraic decaying initial
  data}, J. Math. Anal. Appl., 448 (2017), pp.~1--21.

\bibitem{Henderson2016}
{\sc C.~Henderson}, {\em Propagation of solutions to the {F}isher-{KPP}
  equation with slowly decaying initial data}, Nonlinearity, 29 (2016),
  pp.~3215--3240.

\bibitem{HLWW}
{\sc Q.~Hou, C.~J. Liu, Y.~G. Wang, and Z.-A. Wang}, {\em Stability of boundary
  layers for a viscous hyperbolic system arising from chemotaxis: one
  dimensional case}, SIAM J. Math. Anal., 50 (2018), pp.~3058--3091.

\bibitem{hou2019convergence}
{\sc Q.~Hou and Z.-A. Wang}, {\em Convergence of boundary layers for the
  {K}eller--{S}egel system with singular sensitivity in the half-plane}, J.
  Math. Pures Appl., 130 (2019), pp.~251--287.

\bibitem{IssaShen2017}
{\sc T.~B. Issa and W.~Shen}, {\em Dynamics in chemotaxis models of
  parabolic-elliptic type on bounded domain with time and space dependent
  logistic sources}, SIAM J. Appl. Dyn. Syst., 16 (2017), pp.~926--973.

\bibitem{Odell75}
{\sc E.~Keller and G.~Odell}, {\em Necessary and sufficient conditions for
  chemotactic bands}, Math. Biosci., 27 (1975), pp.~309--317.

\bibitem{KS-tws}
{\sc E.~Keller and L.~Segel}, {\em Traveling bands of chemotactic bacteria: a
  theoretical analysis}, J. Theor. Biol., 30 (1971), pp.~235--248.

\bibitem{kolmogorov1937study}
{\sc A.~N. Kolmogorov, I.~G. Petrovsky, and N.~S. Piskunov}, {\em \'{E}tude de
  l'\'{e}quation de la diffusion avec croissance de la quantit\'{e} de
  mati\`{e}re et son application \`{a} un probl\`{e}me biologique}, Bull. Univ.
  \'{E}tat Moscou, S\'{e}r. Internat., A1 (1937), pp.~1--26.

\bibitem{landman2005diffusive}
{\sc K.~Landman, M.~Simpson, J.~Slater, and D.~Newgreen}, {\em Diffusive and
  chemotactic cellular migration: smooth and discontinuous traveling wave
  solutions}, SIAM J. Appl. Math., 65 (2005), pp.~1420--1442.

\bibitem{lapidus1978model}
{\sc I.~Lapidus and R.~Schiller}, {\em A model for traveling bands of
  chemotactic bacteria}, Biophysical J., 22 (1978), pp.~1--13.

\bibitem{Lau1985}
{\sc K.-S. Lau}, {\em On the nonlinear diffusion equation of {K}olmogorov,
  {P}etrovsky, and {P}iscounov}, J. Differential Equations, 59 (1985),
  pp.~44--70.

\bibitem{lauffenburger1984traveling}
{\sc D.~Lauffenburger, C.~Kennedy, and R.~Aris}, {\em Traveling bands of
  chemotactic bacteria in the context of population growth}, Bull. Math. Biol.,
  46 (1984), pp.~19--40.

\bibitem{LLW}
{\sc J.~Li, T.~Li, and Z.-A. Wang}, {\em Stability of traveling waves of the
  {K}eller-{S}egel system with logarithmic sensitivity}, Math. Models Methods
  Appl. Sci., 24 (2014), pp.~2819--2849.

\bibitem{LPZ}
{\sc T.~Li, R.~Pan, and K.~Zhao}, {\em Global dynamics of a
  hyperbolic-parabolic model arising from chemotaxis}, SIAM J.\ Appl.\ Math.,
  72 (2012), pp.~417--443.

\bibitem{LW09}
{\sc T.~Li and Z.-A. Wang}, {\em Nonlinear stability of travelling waves to a
  hyperbolic-parabolic system modeling chemotaxis}, SIAM J.\ Appl.\ Math., 70
  (2009), pp.~1522--1541.

\bibitem{LiangZhao2007}
{\sc X.~Liang and X.-Q. Zhao}, {\em Asymptotic speeds of spread and traveling
  waves for monotone semiflows with applications}, Comm. Pure Appl. Math., 60
  (2007), pp.~1--40.

\bibitem{LiangZhao2010}
{\sc X.~Liang and X.-Q. Zhao}, {\em Spreading speeds and traveling waves for
  abstract monostable evolution systems}, J. Funct. Anal., 259 (2010),
  pp.~857--903.

\bibitem{liu2011sequential}
{\sc C.~Liu, X.~Fu, X.~Liu, L.and~Ren, K.~Chau, S.~Li, L.~Xiang, H.~Zeng,
  G.~Chen, and L.-H. Tang}, {\em Sequential establishment of stripe patterns in
  an expanding cell population}, Science, 334 (2011), pp.~238--241.

\bibitem{Lui1989}
{\sc R.~Lui}, {\em Biological growth and spread modeled by systems of
  recursions. {I}. {M}athematical theory}, Math. Biosci., 93 (1989),
  pp.~269--295.

\bibitem{LuiWang}
{\sc R.~Lui and Z.-A. Wang}, {\em Traveling wave solutions from microscopic to
  macroscopic chemotaxis models}, J. Math. Biol., 61 (2010), pp.~739--761.

\bibitem{MWZ}
{\sc V.~Martinez, Z.-A. Wang, and K.~Zhao}, {\em Asymptotic and viscous
  stability of large-amplitude solutions of a hyperbolic system arising from
  biology}, {\it Indiana Univ. Math. J. }, 67 (2018), pp.~1383--1424.

\bibitem{McKean1975}
{\sc H.~P. McKean}, {\em Application of {B}rownian motion to the equation of
  {K}olmogorov-{P}etrovskii-{P}iskunov}, Comm. Pure Appl. Math., 28 (1975),
  pp.~323--331.

\bibitem{ryzhik2008traveling}
{\sc G.~Nadin, B.~Perthame, and L.~Ryzhik}, {\em Traveling waves for the
  {K}eller--{S}egel system with fisher birth terms}, Interfaces Free Bound., 10
  (2008), pp.~517--538.

\bibitem{NagaiIkeda}
{\sc T.~Nagai and T.~Ikeda}, {\em Traveling waves in a chemotactic model}, J.
  Math. Biol., 30 (1991), pp.~169--184.

\bibitem{narla2021traveling}
{\sc A.~Narla, J.~Cremer, and T.~Hwa}, {\em A traveling-wave solution for
  bacterial chemotaxis with growth}, arXiv:2103.08100,  (2021).

\bibitem{OsakiTsujikawaYagiMimura2002}
{\sc K.~Osaki, T.~Tsujikawa, A.~Yagi, and M.~Mimura}, {\em Exponential
  attractor for a chemotaxis-growth system of equations}, Nonlinear Anal., 51
  (2002), pp.~119--144.

\bibitem{painter2011spatio}
{\sc K.~Painter and T.~Hillen}, {\em Spatio-temporal chaos in a chemotaxis
  model}, Physica D: Nonlinear Phenomena, 240 (2011), pp.~363--375.

\bibitem{Salako-DCDS-2017}
{\sc R.~Salako and W.~Shen}, {\em Spreading speeds and traveling waves of a
  parabolic-elliptic chemotaxis system with logistic source on $\mathbb{R}^n$},
  Discrete Contin. Dyn. Syst., 37 (2017), pp.~61--89.

\bibitem{salako2018existence}
{\sc R.~Salako and W.~Shen}, {\em Existence of traveling wave solutions of
  parabolic--parabolic chemotaxis systems}, Nonlinear Analysis: Real World
  Applications, 42 (2018), pp.~93--119.

\bibitem{salako2020traveling}
{\sc R.~Salako and W.~Shen}, {\em Existence of traveling wave solutions of
  parabolic--parabolic chemotaxis systems}, Electron. J. Differ. Equ, 2020
  (2020), pp.~1--18.

\bibitem{SS2017}
{\sc R.~B. Salako and W.~Shen}, {\em Global existence and asymptotic behavior
  of classical solutions to a parabolic-elliptic chemotaxis system with
  logistic source on {$\Bbb R^N$}}, J. Differential Equations, 262 (2017),
  pp.~5635--5690.

\bibitem{SSX2019}
{\sc R.~B. Salako, W.~Shen, and S.~Xue}, {\em Can chemotaxis speed up or slow
  down the spatial spreading in parabolic-elliptic {K}eller-{S}egel systems
  with logistic source?}, J. Math. Biol., 79 (2019), pp.~1455--1490.

\bibitem{saragosti2011directional}
{\sc J.~Saragosti, V.~Calvez, N.~Bournaveas, B.~Perthame, A.~Buguin, and
  P.~Silberzan}, {\em Directional persistence of chemotactic bacteria in a
  traveling concentration wave}, Proc. Natl. Acad. Sci. USA, 108 (2011),
  pp.~16235--16240.

\bibitem{Sattinger1976}
{\sc D.~H. Sattinger}, {\em On the stability of waves of nonlinear parabolic
  systems}, Advances in Math., 22 (1976), pp.~312--355.

\bibitem{TelloWinkler2007}
{\sc J.~I. Tello and M.~Winkler}, {\em A chemotaxis system with logistic
  source}, Comm. Partial Differential Equations, 32 (2007), pp.~849--877.

\bibitem{tindall2008overview}
{\sc M.~Tindall, P.~Maini, S.~Porter, and J.~Armitage}, {\em Overview of
  mathematical approaches used to model bacterial chemotaxis ii: bacterial
  populations}, Bull. Math. Biol., 70 (2008), pp.~15--70.

\bibitem{Wang-TWS-DCDSBrev}
{\sc Z.-A. Wang}, {\em Mathematics of traveling waves in chemotaxis - review
  paper}, Disc. Cont. Dyn. Syst.-Series B., 18 (2013), pp.~601--641.

\bibitem{Weinberger1982}
{\sc H.~F. Weinberger}, {\em Long-time behavior of a class of biological
  models}, SIAM J. Math. Anal., 13 (1982), pp.~353--396.

\bibitem{Winkler2010}
{\sc M.~Winkler}, {\em Boundedness in the higher-dimensional
  parabolic-parabolic chemotaxis system with logistic source}, Comm. Partial
  Differential Equations, 35 (2010), pp.~1516--1537.

\end{thebibliography}
\bibliographystyle{mysiam}
\end{document}